%% file: Complete-Reducibility-and-Separability.tex
\documentclass{article}

\usepackage{amssymb, amsmath, amsthm, tikz, graphicx,lmodern,indentfirst,enumerate}

\usepackage[titletoc,page]{appendix}
\usepackage[margin=1.4in]{geometry}

\newtheorem{thm}{Theorem}[section]
\newtheorem{lem}[thm]{Lemma}
\newtheorem{prop}[thm]{Proposition}

\theoremstyle{definition}
\newtheorem{defn}[thm]{Definition}

\theoremstyle{remark}
\newtheorem{rem}[thm]{Remark}

\theoremstyle{definition}

\theoremstyle{definition}
\newtheorem{question}[thm]{Open Problem}

\theoremstyle{remark}
\newtheorem{example}[thm]{Example}

\theoremstyle{definition}

\numberwithin{equation}{section}

\newcommand*\circled[1]{\tikz[baseline=(char.base)]{
            \node[shape=circle,draw,inner sep=0pt,minimum size=5mm] (char) {#1};}}

\newcommand{\rootsG}[8]{\circled{#1}
            \begin{tabular}{cccccc}
             &&&#2&& \\
             #3&#4&#5&#6&#7&#8
             \end{tabular}}

\makeatletter

\newcommand{\Rmnum}[1]{\expandafter\@slowromancap\romannumeral #1@}
\makeatother

\begin{document}

\input{title}
\input{introduction}
\input{preliminaries-a}

\input{preliminaries-b}
\input{E7-example-a}

\input{E7-example-b}

\input{E7-example-c}
\input{E7-example-d}
\input{rationality}
\input{conjugacy}
\input{acknowledgements}
\newpage
\input{Appendix}

\newpage
\bibliography{mybib}

\end{document}

%% file: title.tex
\title{Separability and complete reducibility of subgroups of the Weyl group of a simple algebraic group of type $E_7$}
\author{Tomohiro Uchiyama\\
Department of Mathematics, University of Auckland, \\
Private Bag 92019, Auckland 1142, New Zealand\\
\texttt{email:tuch540@aucklanduni.ac.nz}}
\date{}
\maketitle

\begin{abstract}
Let $G$ be a connected reductive algebraic group defined over an algebraically closed field $k$. The aim of this paper is to present a method to find triples $(G,M,H)$ with the following three properties. 
Property~1: $G$ is simple and $k$ has characteristic $2$. 
Property~2: $H$ and $M$ are closed reductive subgroups of $G$ such that $H<M<G$, and $(G, M)$ is a reductive pair.  
Property~3: $H$ is $G$-completely reducible, but not $M$-completely reducible.  
We exhibit our method by presenting a new example of such a triple in $G=E_7$. Then we consider a rationality problem and a problem concerning conjugacy classes as important applications of our construction. 
\end{abstract}


%% file: introduction.tex
\section{Introduction}

Let $G$ be a connected reductive algebraic group defined over an algebraically closed field $k$ of characteristic $p$. In ~\cite[Sec.~3]{Serre-building}, J.P. Serre defined that a closed subgroup $H$ of $G$ is \emph{$G$-completely reducible} ($G$-cr for short) if whenever $H$ is contained in a parabolic subgroup $P$ of $G$, $H$ is contained in a Levi subgroup $L$ of $P$. This is a faithful generalization of the notion of semisimplicity in representation theory since if $G=GL_n(k)$, a subgroup $H$ of $G$ is $G$-cr if and only if $H$ acts complete reducibly on $k^n$~\cite[Ex.~3.2.2(a)]{Serre-building}. It is known that if a closed subgroup $H$ of $G$ is $G$-cr, then $H$ is reductive~\cite[Prop.~4.1]{Serre-building}. Moreover, if $p=0$, the converse holds~\cite[Prop.~4.2]{Serre-building}. Therefore the notion of $G$-complete reducibility is not interesting if $p=0$. In this paper, we assume that $p>0$. 

Completely reducible subgroups of connected reductive algebraic groups have been much studied \cite{Liebeck-Seitz}, \cite{Liebeck-Testerman}, \cite{Serre-building}. Recently, studies of complete reducibility via Geometric Invariant Theory (GIT for short) have been fruitful \cite{Ben2}, \cite{Ben3}, \cite{Ben1}. In this paper, we see another application of GIT to  complete reducibility (Proposition~\ref{H is not M-cr}). 

Here is the main problem we consider. Let $H$ and $M$ be closed reductive subgroups of $G$ such that $H\leq M\leq G$. It is natural to ask
whether $H$ being $M$-cr implies that $H$ is $G$-cr and vice versa. It is not difficult to find a counterexample for the forward direction. For example, take $H=M=PGL_2(k)$ and $G=SL_3(k)$ where $p=2$ and $H$ sits inside $G$ via the adjoint representation. Another such example is \cite[Ex.~3.45]{Ben2}. However, it is hard to get a counterexample for the reverse direction, and it necessarily involves a small $p$. In \cite[Sec.~7]{Ben1}, Bate et al.~presented the only known counterexample for the reverse direction where $p=2$, $H\cong S_3$, $M\cong A_1 A_1$, and $G=G_2$, which we call ``the $G_2$ example''. The aim of this paper is to prove the following.

\begin{thm}\label{G-cr but not M-cr}
Let $G$ be a simple algebraic group of type $E_7$ defined over $k$ of characteristic $p=2$. Then there exists a connected reductive subgroup $M$ of type $A_7$ of $G$ and a reductive subgroup $H\cong D_{14}$ (the dihedral group of order $14$) of $M$ such that $(G,M)$ is a reductive pair and $H$ is $G$-cr but not $M$-cr. 
\end{thm} 

Our work is motivated by \cite{Ben1}. We recall a few relevant definitions and results here. We denote the Lie algebra of $G$ by $\rm{Lie}\,G=\mathfrak{g}$. From now on, by a subgroup of $G$, we always mean a closed subgroup of $G$.

\begin{defn}
Let $H$ be a subgroup of $G$ acting on $G$ by inner automorphisms. Let $H$ act on $\mathfrak{g}$ by the corresponding adjoint action. Then $H$ is called \emph{separable} if ${\rm Lie}\,C_G(H) = \mathfrak{c}_{\mathfrak{g}}(H)$.
\end{defn}

\noindent Recall that we always have ${\rm Lie}\,C_G(H) \subseteq c_{\mathfrak{g}}(H)$. In \cite{Ben1}, Bate et al.~investigated the relationship between $G$-complete reducibility and separability, and showed the following  \cite[Thm.~1.2, Thm.~1.4]{Ben1}.

\begin{prop}\label{p-good-separability}
Suppose that $p$ is very good for $G$. Then any subgroup of $G$ is separable in $G$.
\end{prop}
\begin{prop}\label{reductive-pair}
Suppose that $(G, M)$ is a reductive pair. Let $H$ be a subgroup of $M$ such that $H$ is a separable subgroup of $G$. If $H$ is $G$-cr, then it is also $M$-cr. 
\end{prop}
\noindent Recall that a pair of reductive groups $G$ and $M$ is called a \emph{reductive pair}
if ${\rm Lie}\,M$ is an $M$-module direct summand of $\mathfrak{g}$.
This is automatically satisfied if $p=0$.  
Propositions~\ref{p-good-separability} and~\ref{reductive-pair} imply that the subgroup $H$ in Theorem~\ref{G-cr but not M-cr} must be non-separable, which is possible for small $p$ only.

Now, we introduce the key notion of \emph{separable action}, which is a slight generalization of the notion of a separable subgroup. 
\begin{defn}
Let $H$ and $N$ be subgroups of $G$ where $H$ acts on $N$ by group automorphisms. The action of $H$ is called \emph{separable} in $N$ if ${\rm Lie}\,C_N(H) =  \mathfrak{c}_{{\rm Lie}\,N}(H)$. Note that the condition means that the fixed points of $H$ acting on $N$, taken with their natural scheme structure, are smooth. 
\end{defn}
\noindent Here is a brief sketch of our method. {\bf Note that in our construction, $p$ needs to be $2$.}
\begin{enumerate}
\item Pick a parabolic subgroup $P$ of $G$ with a Levi subgroup $L$ of $P$. Find a subgroup $K$ of $L$ such that $K$ acts non-separably on the unipotent radical $R_u(P)$ of $P$. In our case, $K$ is generated by elements corresponding to certain reflections in the Weyl group of $G$.
\item Conjugate $K$ by a suitable element $v$ of $R_u(P)$, and set $H=vKv^{-1}$. Then choose a connected reductive subgroup $M$ of $G$ such that $H$ is not $M$-cr. Use a recent result from GIT (Proposition~\ref{GIT}) to show that $H$ is not $M$-cr. Note that $K$ is $M$-cr in our case. 
\item Prove that $H$ is $G$-cr. 
\end{enumerate}
\begin{rem}
It can be shown using \cite[Thm.~13.4.2]{Springer} that $K$ in Step~1 is a non-separable subgroup of $G$. 
\end{rem}

First of all, for Step~1, $p$ cannot be very good for $G$ by Proposition~\ref{p-good-separability} and \ref{reductive-pair}. It is known that $2$ and $3$ are bad for $E_7$. We explain the reason why we choose $p=2$, not $p=3$ (Remark~\ref{2 is important}). Remember that the non-separable action on $R_u(P)$ was the key ingredient for the $G_2$ example to work. Since $K$ is isomorphic to a subgroup of the Weyl group of $G$, we are able to turn a problem of non-separability into a purely combinatorial problem involving the root system of $G$ (Section~3.1). 
Regarding Step~2, we explain the reason of our choice of $v$ and $M$ explicitly (Remarks~\ref{choice of v},~\ref{E7M}). Our use of  Proposition~\ref{GIT} gives an improved way for checking $G$-complete reducibility (Remark~\ref{reason}).
Finally, Step~3 is easy.

In the $G_2$ and $E_7$ examples, the $G$-cr and non-$M$-cr subgroups $H$ are finite. The following is the only known example of a triple $(G,M,H)$ with positive dimensional $H$ such that $H$ is $G$-cr but not $M$-cr. It is obtained by modifying \cite[Ex.~3.45]{Ben2}.
\begin{example}
Let $p=2$, $m\geq 4$ be even, and $(G,M)=(GL_{2m}(k),Sp_{2m}(k))$. Let $H$ be a copy of $Sp_{m}(k)$ diagonally embedded in $Sp_m(k)\times Sp_m(k)$. Then $H$ is not $M$-cr by the argument in \cite[Ex.~3.45]{Ben2}. But $H$ is $G$-cr since $H$ is $GL_m(k)\times GL_m(k)$-cr by \cite[Lem.~2.12]{Ben2}. Also note that any subgroup of $GL(k)$ is separable in $GL(k)$ (cf.~\cite[Ex.~3.28]{Ben2}), so $(G,M)$ is not a reductive pair by Proposition~\ref{reductive-pair}. 
\end{example}

In view of this, it is natural to ask: 
\begin{question}
Is there a triple $H<M<G$ of connected reductive algebraic groups such that $(G,M)$ is a reductive pair, $H$ is non-separable in $G$, and $H$ is $G$-cr but not $M$-cr?
\end{question}

Beyond its intrinsic interest, our $E_7$ example has some important consequences and applications. For example, in Section 6, we consider a rationality problem concerning complete reducibility. We need a definition first to explain our result there. 
\begin{defn}
Let $k_0$ be a subfield of an algebraically closed field $k$. Let $H$ be a $k_0$-defined closed subgroup of a $k_0$-defined reductive algebraic group $G$. Then $H$ is called \emph{$G$-cr over $k_0$} if whenever $H$ is contained in a $k_0$-defined parabolic subgroup $P$ of $G$, it is contained in some $k_0$-defined Levi subgroup of $P$. 
\end{defn}
Note that if $k_0$ is algebraically closed then $G$-cr over $k_0$ means $G$-cr in the usual sense. Here is the main result of Section 6. 
\begin{thm}\label{rationality}
Let $k_0$ be a nonperfect field of charecteristic $p=2$, and let $G$ be a $k_0$-defined split simple algebraic group of type $E_7$. Then there exists a $k_0$-defined subgroup $H$ of $G$ such that $H$ is $G$-cr over $k$, but not $G$-cr over $k_0$. 
\end{thm}

As another application of the $E_7$ example, we consider a problem concerning conjugacy classes. Given $n\in {\mathbb N}$, we let $G$ act on $G^n$ by simultaneous conjugation:
\begin{equation*}
g\cdot(g_1, g_2, \ldots, g_n) = (g g_1 g^{-1}, g g_2 g^{-1}, \ldots, g g_n g^{-1}). 
\end{equation*}
 In \cite{Slodowy}, Slodowy proved the following fundamental result applying Richardson's tangent space argument,~\cite[Sec.~3]{Richardson2},~\cite[Lem.~3.1]{Richardson3}. 
\begin{prop}\label{conjugacy}
Let $M$ be a reductive subgroup of a reductive algebraic group $G$ defined over $k$. Let $n\in {\mathbb N}$, let $(m_1, \ldots, m_n)\in M^n$ and let $H$ be the subgroup of $M$ generated by $m_1, \ldots, m_n$. Suppose that $(G, M)$ is a reductive pair and that $H$ is separable in $G$. Then the intersection $G\cdot (m_1, \ldots, m_n)\cap M^n$ is a finite union of $M$-conjugacy classes. 
\end{prop}
Proposition~\ref{conjugacy} has many consequences. See \cite{Ben2}, \cite{Slodowy}, and \cite[Sec.~3]{Vinberg} for example. In \cite[Ex.~7.15]{Ben1}, Bate et al.~found a counterexample for $G=G_2$ showing that Proposition~\ref{conjugacy} fails without the separability hypothesis. In Section 7, we present a new counterexample to Proposition~\ref{conjugacy} without the separability hypothesis. Here is the main result of Section 7.

\begin{thm}\label{conjugacy-counterexample}
Let $G$ be a simple algebraic group of type $E_7$ defined over an algebraically closed $k$ of characteristic $p=2$. Let $M$ be the connected reductive subsystem subgroup of type $A_7$. Then there exists $n\in \mathbb{N}$ and a tuple $\mathbf{m}\in M^n$ such that $G\cdot \bold{m} \cap M^n$ is an infinite union of $M$-conjugacy classes. Note that $(G,M)$ is a reductive pair in this case.
\end{thm} 

Now, we give an outline of the paper. 
In Section 2, we fix our notation which follows \cite{Borel}, \cite{Humphreys}, and \cite{Springer}. Also, we recall some preliminary results, in particular, Proposition \ref{GIT} from GIT. After that, in Section 3, we prove our main result, Theorem~\ref{G-cr but not M-cr}. Then in Section 4, we consider a rationality problem, and prove Theorem~\ref{rationality}. Finally, in Section 5, we discuss a problem concerning conjugacy classes, and prove Theorem~\ref{conjugacy-counterexample}.  
   

%% file: preliminaries-a.tex
\section{Preliminaries}
\subsection{Notation}
Throughout the paper, we denote by $k$ an algebraically closed field of positive characteristic $p$. We denote the multiplicative group of $k$ by $k^*$. 
We use a capital roman letter, $G$, $H$, $K$, etc., to represent an algebraic group, and the corresponding lowercase gothic letter, $\mathfrak{g}$, $\mathfrak{h}$, $\mathfrak{k}$, etc., to represent its Lie algebra. We sometimes use another notation for Lie algebras:
$\text{Lie}\,G$, $\text{Lie}\,H$, and $\text{Lie}\,K$ are the Lie algebras of $G$, $H$, and $K$ respectively. 

We denote the identity component of $G$ by $G^{\circ}$. We write $[G,G]$ for the derived group of $G$.
The \emph{unipotent radical} of $G$ is denoted by $R_u(G)$. An algebraic group $G$ is \emph{reductive} if $R_u(G)=\{1\}$. In particular, $G$ is \emph{simple} as an algebraic group if $G$ is connected and all proper normal subgroups of $G$ are finite.

In this paper, when a subgroup $H$ of $G$ acts on $G$, $H$ always acts on $G$ by inner automorphisms. The adjoint representation of $G$ is denoted by $\text{Ad}_{\mathfrak{g}}$ or just Ad if no confusion arises. We write $C_G(H)$ and $\mathfrak{c}_{\mathfrak{g}}(H)$ for the global and the infinitesimal centralizers of $H$ in $G$ and $\mathfrak{g}$ respectively. We write $X(G)$ and $Y(G)$ for the set of characters and cocharacters of $G$ respectively. 

\subsection{Complete reducibility and GIT}
Let $G$ be a connected reductive algebraic group. We recall Richardson's formalism \cite[Sec.~2.1--2.3]{Richardson} for the characterization of a parabolic subgroup $P$ of $G$, a Levi subgroup $L$ of $P$, and the unipotent radical $R_u(P)$ of $P$ in terms of a cocharacter of $G$ and state a result from GIT (Proposition~\ref{GIT}).  

\begin{defn}
Let $X$ be an affine variety. Let $\phi : k^*\rightarrow X$ be a morphism of algebraic varieties. We say that $\displaystyle\lim_{a\rightarrow 0}\phi(a)$ exists if there exists a morphism $\hat\phi: k\rightarrow X$ (necessarily unique) whose restriction to $k^{*}$ is $\phi$. If this limit exists, we set $\displaystyle\lim_{a\rightarrow 0}\phi(a) = \hat\phi(0)$.
\end{defn}

\begin{defn}
Let $\lambda$ be a cocharacter of $G$. Define
$
P_\lambda := \{ g\in G \mid \displaystyle\lim_{a\rightarrow 0}\lambda(a)g\lambda(a)^{-1} \text{ exists}\}, $\\
$L_\lambda := \{ g\in G \mid \displaystyle\lim_{a\rightarrow 0}\lambda(a)g\lambda(a)^{-1} = g\}, \,
R_u(P_\lambda) := \{ g\in G \mid  \displaystyle\lim_{a\rightarrow0}\lambda(a)g\lambda(a)^{-1} = 1\}. 
$
\end{defn}
Note that $P_\lambda$ is a parabolic subgroup of $G$, $L_\lambda$ is a Levi subgroup of $P_\lambda$, and $R_u(P_\lambda)$ is a unipotent radical of $P_\lambda$~\cite[Sec.~2.1-2.3]{Richardson}. By~\cite[Prop.~8.4.5]{Springer}, any parabolic subgroup $P$ of $G$, any Levi subgroup $L$ of $P$, and any unipotent radical $R_u(P)$ of $P$ can be expressed in this form. It is well known that $L_\lambda = C_G(\lambda(k^*))$. 

Let $M$ be a reductive subgroup of $G$. Then, there is a natural inclusion $Y(M)\subseteq Y(G)$ of cocharacter groups. Let $\lambda\in Y(M)$. We write $P_\lambda(G)$ or just $P_\lambda$ for the parabolic subgroup of $G$ corresponding to $\lambda$, and $P_\lambda(M)$ for the parabolic subgroup of $M$ corresponding to $\lambda$. It is obvious that $P_\lambda(M) = P_\lambda(G)\cap M$ and $R_u(P_\lambda(M)) = R_u(P_\lambda(G))\cap M$. 

\begin{defn}\label{homomorphism}
Let $\lambda\in Y(G)$. Define a map $c_\lambda : P_\lambda \rightarrow L_\lambda$ by 
$
c_\lambda(g) := \displaystyle\lim_{a\rightarrow 0} \lambda(a) g \lambda(a)^{-1}.
$
\end{defn}
Note that the map $c_\lambda$ is the usual canonical projection from $P_\lambda$ to $L_\lambda \cong P_\lambda / R_u(P_\lambda)$.
Now, we state a result from GIT (see~\cite[Lem.~2.17, Thm.~3.1]{Ben2},~\cite[Thm.~3.3]{Ben3}).
\begin{prop}\label{GIT}
Let $H$ be a subgroup of $G$. Let $\lambda$ be a cocharacter of $G$ with $H\subseteq P_\lambda$. If $H$ is $G$-cr, there exists $v\in R_u(P_\lambda)$ such that $c_\lambda(h) = vhv^{-1}$ for every $h\in H$. 
\end{prop}

%% file: preliminaries-b.tex
\subsection{Root subgroups and root subspaces}
Let $G$ be a connected reductive algebraic group. Fix a maximal torus $T$ of $G$. Let $\Psi(G,T)$ denote the set of roots of $G$ with respect to $T$. We sometimes write $\Psi(G)$ for $\Psi(G,T)$. Fix a Borel subgroup $B$ containing $T$. Then $\Psi(B,T) = \Psi^+(G)$ is the set of positive roots of $G$ defined by $B$. Let $\Sigma(G,B) = \Sigma$ denote the set of simple roots of $G$ defined by $B$. Let $\zeta\in\Psi(G)$. We write $U_\zeta$ for the corresponding root subgroup of $G$ and $\mathfrak{u}_\zeta$ for the Lie algebra of $U_\zeta$. We define $G_\zeta := \langle U_\zeta, U_{-\zeta} \rangle$.  

Let $H$ be a subgroup of $G$ normalized by some maximal torus $T$ of $G$. Consider the adjoint representation of $T$ on $\mathfrak{h}$. The root spaces of $\mathfrak{h}$ with respect to $T$ are also root spaces of $\mathfrak{g}$ with respect to $T$, and the set of roots of $H$ relative to $T$, $\Psi(H,T) = \Psi(H) = \{ \zeta\in \Psi(G) \mid \mathfrak{g}_\zeta \subseteq \mathfrak{h} \}$, is a subset of $\Psi(G)$. 

Let $\zeta, \xi \in \Psi(G)$. Let $\xi^{\vee}$ be the coroot corresponding to $\xi$. Then $\zeta\circ\xi^{\vee}:k^{*}\rightarrow k^{*}$ is a homomorphism such that $(\zeta\circ\xi^{\vee})(a) = a^n$ for some $n\in\mathbb{Z}$. We define $\langle \zeta, \xi^{\vee} \rangle := n$.
Let $s_\xi$ denote the reflection corresponding to $\xi$ in the Weyl group of $G$. Each $s_\xi$ acts on the set of roots $\Psi(G)$ by the following formula~\cite[Lem.~7.1.8]{Springer}:
$
s_\xi\cdot\zeta = \zeta - \langle \zeta, \xi^{\vee} \rangle \xi. 
$
\noindent By \cite[Prop.~6.4.2, Lem.~7.2.1]{Carter}, we can choose homomorphisms $\epsilon_\zeta : k \rightarrow U_\zeta$  so that 
\begin{equation}
n_\xi \epsilon_\zeta(a) n_\xi^{-1}= \epsilon_{s_\xi\cdot\zeta}(\pm a), 
            \text{ where } n_\xi = \epsilon_\xi(1)\epsilon_{-\xi}(-1)\epsilon_{\xi}(1).  \label{n-action on group}
\end{equation}
We define $e_\zeta:=\epsilon_\zeta'(0)$. Then we have 
\begin{equation}
\text{Ad}(n_\xi) e_\zeta = \pm e_{s_\xi\cdot\zeta}. \label{n-action on Lie algebra}
\end{equation}

Now, we list four lemmas which we need in our calculations. 
The first one is~\cite[Prop.~8.2.1]{Springer}.
\begin{lem}\label{unipotent radical lemma}
Let $P$ be a parabolic subgroup of $G$. Any element $u$ in $R_u(P)$ can be expressed uniquely as 
\begin{equation*}
u = \prod_{i\in \Psi\left(R_u(P)\right)} \epsilon_i(a_i), \textup{ for some } a_i\in k, 
\end{equation*}
where the product is taken with respect to a fixed ordering of $\Psi\left(R_u(P)\right)$. 
\end{lem}
The next two lemmas~\cite[Lem.~32.5 and  Lem.~33.3]{Humphreys} are used to calculate $C_{R_u(P)}(K)$.
\begin{lem}\label{commutation lemma}
Let $\xi, \zeta \in \Psi(G)$. If no positive integral linear combination of $\xi$ and $\zeta$ is a root of $G$, then
\begin{equation*}
\epsilon_\xi(a)\epsilon_\zeta(b) = \epsilon_\zeta(b)\epsilon_\xi(a).
\end{equation*}
\end{lem}
\begin{lem}\label{A_2 commutation lemma}
Let $\Psi$ be the root system of type $A_2$ spanned by roots $\xi$ and $\zeta$. Then
\begin{equation*}
\epsilon_\xi(a)\epsilon_\zeta(b) = \epsilon_\zeta(b)\epsilon_\xi(a)\epsilon_{\xi+\zeta}(\pm ab).
\end{equation*}
\end{lem}
The last result is used to calculate $\mathfrak{c}_{\text{Lie}\,(R_u(P))}(K)$. 
\begin{lem}\label{infinitesimal centralizer lemma}
Suppose that $p=2$. Let $W$ be a subgroup of $G$ generated by all the $n_\xi$ where $\xi\in \Psi(G)$ (the group $W$ is isomorphic to the Weyl group of $G$). Let $K$ be a subgroup of $W$. Let $\{ O_i \mid i = 1\cdots m \}$ be the set of orbits of the action of $K$ on $\Psi\left(R_u(P)\right)$. Then, 
\begin{equation*}
\mathfrak{c}_{\textup{Lie}\left(R_u(P)\right)}(K) = \left\{ \left.\sum_{i=1}^{m}a_i \sum_{\zeta\in O_i} e_\zeta \right| a_i \in k \right\}.
\end{equation*}   
\begin{proof}
When $p=2$, (\ref{n-action on Lie algebra}) yields
$
\textup{Ad}(n_\xi) e_\zeta = e_{n_\xi\cdot \zeta}.
$
Then an easy calculation gives the desired result. 
 \end{proof}
\end{lem}
\begin{rem}\label{2 is important}
Lemma~\ref{infinitesimal centralizer lemma} holds in $p=2$ but fails in $p=3$. 
\end{rem}

%% file: E7-example-a.tex
\section{The $E_7$ example}
\subsection{Step~1}
Let $G$ be a simple algebraic group of type $E_7$ defined over $k$ of characteristic $2$. 
Fix a maximal torus $T$ of $G$. Fix a Borel subgroup $B$ of $G$ containing $T$. 
Let $\Sigma = \{ \alpha, \beta, \gamma, \delta, \epsilon, \eta, \sigma \}$ be the set of simple roots of $G$. Figure~\ref{Dynkin diagram of E_7} defines how each simple root of $G$ corresponds to each node in the Dynkin diagram of $E_7$. 

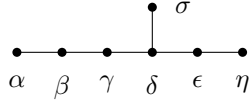
\begin{figure}[!h]
                \begin{center}
                \begin{tikzpicture}[scale=0.6]
                \draw (0,0) to (1,0);
                \draw (1,0) to (2,0);
                \draw (2,0) to (3,0);
                \draw (3,0) to (4,0);
                \draw (4,0) to (5,0);
                \draw (3,0) to (3,1);
                \fill (0,0) circle (1mm);
                \fill (1,0) circle (1mm);
                \fill (2,0) circle (1mm);
                \fill (3,0) circle (1mm);
                \fill (4,0) circle (1mm);
                \fill (5,0) circle (1mm);
                \fill (3,1) circle (1mm);
                \draw[below] (0,-0.3) node{$\alpha$};
                \draw[below] (1,-0.3) node{$\beta$};
                \draw[below] (2,-0.3) node{$\gamma$};
                \draw[below] (3,-0.3) node{$\delta$};
                \draw[below] (4,-0.3) node{$\epsilon$};
                \draw[below] (5,-0.3) node{$\eta$};
                \draw[right] (3.3,1) node{$\sigma$};
                \end{tikzpicture}
                \end{center}
                \caption{Dynkin diagram of $E_7$}
                \label{Dynkin diagram of E_7}
\end{figure}

From~\cite[Appendix, Table~B]{Freudenthal}, one knows the coefficients of all positive roots of $G$. We label all positive roots of $G$ in Table~\ref{Roots of E_7} in the Appendix. Our ordering of roots is different from ~\cite[Appendix, Table~B]{Freudenthal}, which will be convenient later on. 

The set of positive roots is 
$
\Psi^{+}(G) = \{ 1, 2, \cdots, 63 \}.
$
Note that $\{1, \cdots, 35\}$ and $\{36, \cdots, 42\}$ are precisely the roots of $G$ such that the coefficient of $\sigma$ is $1$ and $2$ respectively. We call the roots of the first type \emph{weight-1 roots}, and the second type \emph{weight-2 roots}. 
Define
\begin{equation*}
L_{\alpha\beta\gamma\delta\epsilon\eta} := \langle T, G_{43}, \cdots, G_{63} \rangle, \,
P_{\alpha\beta\gamma\delta\epsilon\eta} := \langle L_{\alpha\beta\gamma\delta\epsilon\eta}, U_{1}, \cdots, U_{42} \rangle.
\end{equation*}
Then $P_{\alpha\beta\gamma\delta\epsilon\eta}$ is a parabolic subgroup of $G$, and $L_{\alpha\beta\gamma\delta\epsilon\eta}$ is a Levi subgroup of $P_{\alpha\beta\gamma\delta\epsilon\eta}$.
Note that $L_{\alpha\beta\gamma\delta\epsilon\eta}$ is of type $A_6$. We have
$
\Psi\left(R_u(P_{\alpha\beta\gamma\delta\epsilon\eta})\right) = \{1,\cdots, 42\}.
$
\noindent Define 
\begin{equation*}
q_1 := n_\epsilon n_\beta n_\gamma n_\alpha n_\beta, \,
q_2 := n_\epsilon n_\beta n_\gamma n_\alpha n_\beta n_\eta n_\delta n_\beta, \,
K := \langle q_1, q_2 \rangle. 
\end{equation*}
Let $\zeta_1, \zeta_2$ be simple roots of $G$. From the Cartan matrix of $E_7$~\cite[Sec.~11.4]{Humphreys2} we have
\begin{equation*}
\langle \zeta_1, \zeta_2 \rangle =
\begin{cases}
2, & \text{if } \zeta_1 = \zeta_2. \\
-1, & \text{if } \zeta_1 \text{ is adjacent to } \zeta_2 \text{ in the Dynkin diagram}. \\ 
0, & \text{otherwise}. 
\end{cases}
\end{equation*}
From this, it is not difficult to calculate $\langle \xi, \zeta^{\vee} \rangle$ for all $\xi \in \Psi\left(R_u(P_{\alpha\beta\gamma\delta\epsilon\eta})\right)$ and for all $\zeta\in \Sigma$. 
These calculations show how $n_\alpha, n_\beta, n_\gamma, n_\delta, n_\epsilon$, and $n_\eta$ act on 
$\Psi\left(R_u(P_{\alpha\beta\gamma\delta\epsilon\eta})\right)$. Let $\pi: \langle n_\alpha, n_\beta, n_\gamma, n_\delta, n_\epsilon, n_\eta \rangle \rightarrow {\rm Sym}\left(\Psi\left(R_u(P_{\alpha\beta\gamma\delta\epsilon\eta})\right)\right)\cong S_{42}$ be the corresponding homomorphism. 
Then we have
\begin{alignat*}{2}
\pi(q_1) =& (1\; 2)(3\; 6)(4\; 7)(9\; 10)(11\; 12)(13\; 14)(15\; 20)
(16\; 17)(18\; 21)(19\; 23)(22\; 25)(24\; 26)\\&(27\; 28)(29\; 32)(31\; 33)(34\; 35)(36\; 38)(37\; 39)(40\; 41), \\
\pi(q_2) =& (1\; 6\; 7\; 5\; 4\; 3\; 2)(8\; 10\; 12\; 14\; 13\; 11\; 9)(15\; 16\; 21\; 23\; 26\; 27\; 22)(17\; 20\; 25\; 28\; 24\; 19\; 18)\\&
(29\; 30\; 32\; 33\; 35\; 34\; 31)(36\; 38\; 39\; 41\; 42\; 40\; 37).
\end{alignat*} 
It is easy to see that $K\cong D_{14}$.
The orbits of $K$ in $\Psi\left(R_u(P_{\alpha\beta\gamma\delta\epsilon\eta})\right)$ are
\begin{alignat*}{2}
O_1 =& \{ 1, \cdots, 7\}, O_8 = \{ 8, \cdots, 14\}, O_{15} = \{ 15, \cdots, 28\}, O_{29} =& \{ 29, \cdots, 35\},\\ O_{36} =& \{ 36, \cdots, 42\}.
\end{alignat*}
Thus Lemma~\ref{infinitesimal centralizer lemma} yields 
\begin{prop}\label{inf-E_7}
\begin{alignat*}{2}
\mathfrak{c}_{\textup{Lie}\left(R_u(P_{\alpha\beta\gamma\delta\epsilon\eta})\right)}(K)=&\Biggr\{ a\left(\sum_{\lambda\in O_1}e_{\lambda}\right)+
b\left(\sum_{\lambda\in O_8}e_{\lambda}\right)+c\left(\sum_{\lambda\in O_{15}}e_{\lambda}\right)+
d\left(\sum_{\lambda\in O_{29}}e_{\lambda}\right)\nonumber\\
&+m\left.\left(\sum_{\lambda\in O_{36}}e_{\lambda}\right)
\right| a,b,c,d,m \in k \Biggr\}. 
\end{alignat*}
\end{prop}

%% file: E7-example-b.tex
The following is the most important technical result in this paper.
\begin{prop}\label{positive proposition}
Let $u\in C_{R_u(P_{\alpha\beta\gamma\delta\epsilon\eta})}(K)$. Then $u$ must have the form, 
\begin{equation*}
u=\prod_{i=1}^{7}\epsilon_i(a)\prod_{i=8}^{14}\epsilon_i(b)\prod_{i=15}^{28}\epsilon_i(c)\prod_{i=29}^{35}
\epsilon_i(a+b+c)\prod_{i=36}^{42}\epsilon_i(a_i) \textup{ for some } a,b,c,a_i\in k.  
\end{equation*}  
\end{prop}
\begin{proof}
By Lemma~\ref{unipotent radical lemma}, $u$ can be expressed uniquely as 
$
u = \prod_{i = 1}^{42} \epsilon_i(b_i) \textup{ for some } b_i\in k.
$
By (\ref{n-action on group}), we have
$
n_\xi \epsilon_\zeta(a) n_\xi^{-1}= \epsilon_{s_\xi\cdot\zeta}(a) \textup{ for any } a\in k \textup{ and }\xi, \zeta \in \Psi(G). 
$ 
Thus we have
\begin{alignat}{2}
q_1 u q_1^{-1}=&\; q_1 \left( \prod_{i=1}^{42} \epsilon_i(b_i)\right) q_1^{-1}\nonumber\\
                        =&\; \left( \prod_{i=1}^{7} \epsilon_{q_1 \cdot i}(b_i)\right)\left(\prod_{i=8}^{14} \epsilon_{q_1 \cdot i}(b_i)\right)\left( \prod_{i=15}^{28} \epsilon_{q_1 \cdot i}(b_i)\right)\left( \prod_{i=29}^{35} \epsilon_{q_1 \cdot i}(b_i)\right)\nonumber\\
                          &\; \left( \prod_{i=36}^{42} \epsilon_{q_1 \cdot i}(b_i)\right).\label{order}
\end{alignat}
A calculation using the commutator relations (Lemma~\ref{commutation lemma} and Lemma~\ref{A_2 commutation lemma}) shows that
\begin{alignat}{2}
q_1 u q_1^{-1} =&\; \epsilon_1(b_2)\epsilon_2(b_1)\epsilon_3(b_6)\epsilon_4(b_7)
\epsilon_5(b_5)\epsilon_6(b_3)\epsilon_7(b_4)\epsilon_8(b_8)\epsilon_9(b_{10})\epsilon_{10}(b_9)\epsilon_{11}(b_{12})
\epsilon_{12}(b_{11})\epsilon_{13}(b_{14})\nonumber\\
&\; \epsilon_{14}(b_{13}) \epsilon_{15}(b_{20})\epsilon_{16}(b_{17})\epsilon_{17}(b_{16})\epsilon_{18}(b_{21})
\epsilon_{19}(b_{23})\epsilon_{20}(b_{15})\epsilon_{21}(b_{18})\epsilon_{22}(b_{25})\epsilon_{23}(b_{19})\epsilon_{24}(b_{26})\nonumber\\
&\; \epsilon_{25}(b_{22})\epsilon_{26}(b_{24})\epsilon_{27}(b_{28})\epsilon_{28}(b_{27}) \epsilon_{29}(b_{32})\epsilon_{30}(b_{30})\epsilon_{31}(b_{33})\epsilon_{32}(b_{29})
\epsilon_{33}(b_{31})\epsilon_{34}(b_{35})\epsilon_{35}(b_{34})\nonumber\\
&\left(\prod_{i=36}^{41} \epsilon_i (a_i)\right) 
\epsilon_{42}(b_4 b_7 + b_{11} b_{12} + b_{22} b_{25} + b_{34} b_{35} + b_{42}) \textup{ for some }a_i\in k. \label{reordered}
\end{alignat}
\noindent Since $q_1$ and $q_2$ centralize $u$, we have 
$
b_1 = \cdots = b_7, \; b_8 = \cdots = b_{14}, \; b_{15} = \cdots = b_{28}, \; b_{29} = \cdots = b_{35}.  
$
Set
$
b_1 = a, \; b_8 = b, \; b_{15} = c, \; b_{29} = d. 
$
Then (\ref{reordered}) simplifies to 
\begin{alignat*}{2}
q_1 u q_1^{-1} & =\prod_{i=1}^{7}\epsilon_i(a)\prod_{i=8}^{14}\epsilon_i(b)\prod_{i=15}^{28}\epsilon_i(c)
\prod_{i=29}^{35}\epsilon_i(d)\left(\prod_{i=36}^{41}\epsilon_i(a_i)\right)\epsilon_{42}(a^2+b^2+c^2+d^2+b_{42}). 
\end{alignat*}
Since $q_1$ centralizes $u$, comparing the arguments of the $\epsilon_{42}$ term on both sides, we must have
\begin{equation*}
b_{42} = a^2 + b^2 + c^2 + d^2 + b_{42},
\end{equation*}
which is equivalent to
$
a + b + c + d = 0.
$
Then we obtain the desired result. 
\end{proof}
\begin{prop}
$K$ acts non-separably on $R_u(P_{\alpha\beta\gamma\delta\epsilon\eta})$.
\end{prop}
\begin{proof}
In view of Proposition~\ref{inf-E_7}, it suffices to show that $e_1+e_2+e_3+e_4+e_5+e_6+e_7\not\in \textup{Lie}\,C_{R_u(P_\lambda)}(K)$. Suppose the contrary. Since by \cite[Cor.~14.2.7]{Springer} $C_{R_u(P_\lambda)}(K)^{\circ}$ is isomorphic as a variety to $k^n$ for some $n\in \mathbb{N}$, there exists a morphism of varieties $v: k\rightarrow C_{R_u(P_\lambda)}(K)^{\circ}$ such that $v(0)=1$ and $v'(0)=e_1+e_2+e_3+e_4+e_5+e_6+e_7$. By Lemma~\ref{unipotent radical lemma}, $v(a)$ can be expressed uniquely as
$
v(a) = \prod_{i=1}^{42}\epsilon_{i}(f_i(a)) 
\text{ for some } f_i \in k[X].
$
 Differentiating the last equation, and evaluating at $a=0$, we obtain
$
v'(0) = \sum_{i\in\{1,\cdots,42\}} (f_i)'(0) e_i.
$
Since $v'(0) = \sum_{i\in O_1}e_i$, we have
\begin{gather*}
(f_i)'(0)=
\begin{cases}
1 &\text{ if } i\in O_1, \\
0 &\text{ otherwise. } 
\end{cases}
\end{gather*}
Then we have
\begin{gather*}
f_i(a) =
\begin{cases}
a + g_i(a) & \text{ if } i\in O_1, \\
g_i(a) & \text{ otherwise, }\\
\end{cases}\\
\text { where } g_i \in k[X] \text{ has no constant or linear term}.\\
\end{gather*}
Then from Proposition~\ref{positive proposition}, we obtain 
$
(a+g_1(a))+g_8(a)+g_{15}(a)=g_{29}(a).
$
This is a contradiction. 
\end{proof}

%% file: E7-example-c.tex
\subsection{Step~2}
Let
$
C_1:= \left\{\prod_{i=1}^7 \epsilon_i(a) \mid a\in k \right\}  
$
, pick any $a\in k^*$, and let $v(a):=\prod_{i=1}^7 \epsilon_i(a)$. Now, set 
\begin{alignat*}{2}
H :=& v(a) K v(a)^{-1} = \langle q_1 \epsilon_{40}(a^2)\epsilon_{41}(a^2)\epsilon_{42}(a^2), q_2 \epsilon_{36}(a^2)\epsilon_{39}(a^2) \rangle,\\
M :=& \langle L_{\alpha\beta\gamma\delta\epsilon\eta}, G_{36},\cdots, G_{42} \rangle. 
\end{alignat*}
\begin{rem}\label{choice of v}
By Proposition~\ref{inf-E_7} and Proposition~\ref{positive proposition}, the tangent space of $C_1$ at the identity, $T_1(C_1)$, is contained in $\mathfrak{c}_{\textup{Lie}(R_u(P_{\alpha\beta\gamma\delta\epsilon\eta}))}(K)$ but not contained in $\textup{Lie}(C_{R_u(P_{\alpha\beta\gamma\delta\epsilon\eta})}(K))$. The element $v(a)$ can be any non-trivial element in $C_1$. 
\end{rem}
\begin{rem}\label{E7M}
In this case $\sigma$ is the unique simple root not contained in $\Psi(L_{\alpha\beta\gamma\delta\epsilon\eta})$. $M$ was chosen so that $M$ is generated by a Levi subgroup $L_{\alpha\beta\gamma\delta\epsilon\eta}$ containing $K$ and all root subgroups of $\sigma$-weight $2$.
\end{rem}
We have 
$
H \subset M, H\not\subset L_{\alpha\beta\gamma\delta\epsilon\eta}.
$
Note that 
$
\Psi(M) = \{ \pm 36, \cdots, \pm 63 \}.
$
Since $M$ is generated by all root subgroups of even $\sigma$-weight, it is easy to see that $\Psi(M)$ is a closed subsystem of $\Psi(G)$, thus $M$ is reductive by~\cite[Lem.~3.9]{Ben1}. Note that $M$ is of type $A_7$.  
\begin{prop}\label{H is not M-cr}
$H$ is not $M$-cr. 
\end{prop}
\begin{proof}
Let 
$
\lambda = 3\alpha^{\vee}+6\beta^{\vee}+9\gamma^{\vee}+12\delta^{\vee}+8\epsilon^{\vee}
+4\eta^{\vee}+7\sigma^{\vee}.
$
We have
\begin{alignat*}{4}
\langle \alpha, \lambda \rangle &= 0, \langle \beta, \lambda \rangle &= 0, \langle \gamma, \lambda \rangle &= 0, \langle \delta, \lambda \rangle &= 0, \\ \langle \epsilon, \lambda \rangle &= 0, \langle \eta, \lambda \rangle &= 0, \langle \sigma, \lambda \rangle &= 2.  
\end{alignat*}
So
$
L_{\alpha\beta\gamma\delta\epsilon\eta} = L_\lambda, \, P_{\alpha\beta\gamma\delta\epsilon\eta} = P_\lambda.
$

It is easy to see that $L_\lambda$ is of type $A_6$, so $[L_\lambda, L_\lambda]$ is isomorphic to either $SL_7$ or $PGL_7$. We rule out the latter. Pick $x\in k^{*}$ such that $x\ne 1, x^7=1$. Then $\lambda(x)\ne 1$ since $\sigma(\lambda(x))=x^2\ne 1$. Also, we have $\lambda(x)\in Z([L_\lambda, L_\lambda])$. Therefore $[L_\lambda, L_\lambda]\cong SL_7$. It is easy to check that the map $k^{*}\times [L_\lambda, L_\lambda]\rightarrow L_\lambda$ is separable, so we have $L_\lambda\cong GL_7$.

Let $c_\lambda : P_\lambda \rightarrow L_\lambda$ be the homomorphism as in Definition~\ref{homomorphism}. 
In order to prove that $H$ is not $M$-cr, by Theorem~\ref{GIT} it suffices to find a tuple $(h_1, h_2)\in H^2$ which is not  $R_u\left(P_\lambda(M)\right)$-conjugate to $c_\lambda\left((h_1, h_2)\right)$. Set 
$
h_1 := v(a) q_1 v(a)^{-1}, \; h_2 := v(a) q_2 v(a)^{-1}. 
$
Then
\begin{alignat*}{2}
c_\lambda\left((h_1, h_2)\right) &= \lim_{x\rightarrow 0}\left(\lambda(x)q_1 \epsilon_{40}(a^2)\epsilon_{41}(a^2)\epsilon_{42}(a^2)\lambda(x)^{-1}, \; (\lambda(x)q_2 \epsilon_{36}(a^2)\epsilon_{39}(a^2)\lambda(x)^{-1}\right)\\
&=(q_1, q_2).
\end{alignat*}   
Now suppose that $(h_1, h_2)$ is $R_u\left(P_\lambda(M)\right)$-conjugate to $c_\lambda\left((h_1, h_2)\right)$. Then there exists $m\in R_u\left(P_\lambda(M)\right)$ such that
\begin{equation*}
 m v(a) q_1 v(a)^{-1}m^{-1} = q_1, \,
 m v(a) q_2 v(a)^{-1}m^{-1} = q_2. \label{m-equation}
\end{equation*}
Thus we have
$
m v(a) \in C_{R_u(P_\lambda)}(K). 
$
Note that 
$
\Psi\left(R_u\left(P_\lambda(M)\right)\right) = \{ 36, \cdots, 42 \}. 
$
So, by Lemma~\ref{unipotent radical lemma}, $m$ can be expressed uniquely as 
$
m := \prod_{i=36}^{42} \epsilon_i(a_i) \textup{ for some }a_i\in k.
$
Then we have
\begin{equation*}\label{last equation}
 m v(a) = \epsilon_1(a) \epsilon_2(a) \epsilon_3(a) \epsilon_4(a) \epsilon_5(a) \epsilon_6(a) \epsilon_7(a)\left( \prod_{i=36}^{42} \epsilon_i(a_i)\right) \in C_{R_u(P_\lambda)}(K). 
\end{equation*}
This contradicts Proposition~\ref{positive proposition}.
\end{proof}
\begin{rem}\label{reason}
In \cite[Sec.~7, Prop~.7.17]{Ben1}, Bate et al.~used  \cite[Lem.~2.17, Thm.~3.1]{Ben2} to turn a problem on $M$-complete reducibility into a problem involving $M$-conjugacy. We have used Proposition~\ref{GIT} to turn the same problem into a problem involving $R_u(P\cap M)$-conjugacy, which is easier.
\end{rem}
\begin{rem}
Instead of using $C_1$ to define $v(a)$, we can take $C_8:= \left\{\prod_{i=8}^{14} \epsilon_i(a) \mid a\in k \right\}$, $C_{15}:= \left\{\prod_{i=15}^{28} \epsilon_i(a) \mid a\in k \right\}$, or $C_{29}:= \left\{\prod_{i=29}^{35} \epsilon_i(a) \mid a\in k \right\}$. In each case, a similar argument goes through and gives rise to a different example with the desired property.
\end{rem}

%% file: E7-example-d.tex
\subsection{Step~3}
\begin{prop}\label{G-cr prop}
$H$ is $G$-cr.
\end{prop}
\begin{proof}
First note that $H$ is conjugate to $K$, so $H$ is $G$-cr if and only if $K$ is $G$-cr. Then, 
by \cite[Lem.~2.12,~Cor.~3.22]{Ben2}, it suffices to show that $K$ is $[L_{\lambda},L_{\lambda}]$-cr. 
We can identify $K$ with the image of the corresponding subgroup of $S_7$ under the permutation representation $\pi_1: S_7 \rightarrow SL_7(k)$. 
It is easy to see that $K\cong D_{14}$. A quick calculation shows that this representation of $D_{14}$ is a direct sum of a trivial $1$-dimensional and $3$ irreducible $2$-dimensional subrepresentations. Therefore $K$ is $[L_{\lambda},L_{\lambda}]$-cr. 
\end{proof}  

%% file: rationality.tex
\section{A rationality problem}
We prove Theorem~\ref{rationality}. The key here is again the existence of a $1$-dimensional curve $C_1$ such that $T_1(C_1)$ is contained in $\mathfrak{c}_{\textup{Lie}\,({R_u(P_\lambda)})}(K)$ but not contained in  $\textup{Lie}(C_{R_u(P_\lambda)}(K))$. The same phenomenon was seen in the $G_2$ example.

\begin{proof}[Proof of Theorem~\ref{rationality}]
Let $k_0$, $k$, and $G$ be as in the hypothesis. We choose a $k_0$-defined $k_0$-split maximal torus $T$ such that for each $\zeta\in\Psi(G)$ the corresponding root $\zeta$, coroot $\zeta^{\vee}$, and homomorphism $\epsilon_\zeta$ are defined over $k_0$. Since $k_0$ is not perfect, there exists $\tilde{a}\in k\backslash k_0$ such that $\tilde{a}^2\in k_0$. We keep the notation $q_1, q_2, v, K, P_\lambda, L_\lambda$ of Section~3. Let
\begin{alignat*}{2}
H &= \langle v(\tilde{a}) q_1 v(\tilde{a})^{-1}, v(\tilde{a}) q_2 v(\tilde{a})^{-1} \rangle \\
       &= \langle q_1 \epsilon_{40}(\tilde{a}^2)\epsilon_{41}(\tilde{a}^2)\epsilon_{42}(\tilde{a}^2), q_2 \epsilon_{36}(\tilde{a}^2)\epsilon_{39}(\tilde{a}^2) \rangle.  
\end{alignat*}
Now it is obvious that $H$ is $k_0$-defined.
We already know that $H$ is $G$-cr by Proposition~\ref{G-cr prop}. Since $G$ and $T$ are $k_0$-split, $P_\lambda$ and 
$L_\lambda$ are $k_0$-defined by~\cite[\Rmnum{5}.20.4, \Rmnum{5}.20.5]{Borel}. Suppose that there exists a $k_0$-Levi subgroup $L'$ of $P_\lambda$ such that $L'$ contains $H$. Then there exists $w\in R_u(P_\lambda)(k_0)$ such that $L' = w L_\lambda w^{-1}$ by~\cite[\Rmnum{5}.20.5]{Borel}. Then $w^{-1} H w \subseteq L_\lambda$ and $v(\tilde{a})^{-1}H v(\tilde{a})
\subseteq L_\lambda$. So we have $c_\lambda(w^{-1} h w) = w^{-1} h w$ and $c_\lambda\left(v(\tilde{a})^{-1} h v(\tilde{a})\right) = v(\tilde{a})^{-1} h v(\tilde{a})$ for any $h\in H$. We also have $c_\lambda(w) = c_\lambda\left(v(\tilde{a})\right)=1$ since $w, v(\tilde{a})\in R_u(P_\lambda)(k)$. 
Therefore we obtain
$
w^{-1}h w = c_\lambda(w^{-1} h w) = c_\lambda(h) = c_\lambda\left(v(\tilde{a})^{-1} h v(\tilde{a})\right) = v(\tilde{a})^{-1} h v(\tilde{a}) 
$
for any $h\in H$. So we have 
$
w = v(\tilde{a}) z \textup{ for some } z\in C_{R_u(P_\lambda)}(K)(k).
$
By Proposition~\ref{positive proposition}, $z$ must have the form
\begin{equation*}
z = \prod_{i=1}^7 \epsilon_i(a) \prod_{i=8}^{14} \epsilon_i(b) \prod_{i=15}^{28}\epsilon_i(c) 
\prod_{i=29}^{35}\epsilon_i(a+b+c) \prod_{i=36}^{42} \epsilon_i(a_i) \textup{ for some } a, b, c, a_i \in k.
\end{equation*}
Then  
\begin{alignat*}{2}
w &= \left( \prod_{i=1}^{7} \epsilon_i(\tilde{a}) \right)\prod_{i=1}^7 \epsilon_i(a) \prod_{i=8}^{14} \epsilon_i(b) \prod_{i=15}^{28}\epsilon_i(c) 
\prod_{i=29}^{35}\epsilon_i(a+b+c) \prod_{i=36}^{42} \epsilon_i(a_i)\\
       &= \prod_{i=1}^{7} \epsilon_i(\tilde{a}+a) \prod_{i=8}^{14} \epsilon_i(b) \prod_{i=15}^{28}\epsilon_i(c) 
\prod_{i=29}^{35}\epsilon_i(a+b+c) \prod_{i=36}^{42} \epsilon_i(b_i) \textup{ for some } b_i\in k. 
\end{alignat*}
Since $w$ is a $k_0$-point, $b$, $c$, and $a+b+c$ all belong to $k_0$, so $a\in k_0$. But $a+\tilde{a}$ belongs to $k_0$ as well, so $\tilde{a}\in k_0$. This is a contradiction.  
\end{proof}
\begin{rem}
As in Section~3, we can take $v(\tilde{a})$ from $C_8$, $C_{15}$, or $C_{29}$. In each case, a similar argument goes through, and gives rise to a different example. 
\end{rem} 
\begin{rem}
\cite[Ex.~5.11]{Ben2} shows that there is a $k_0$-defined subgroup of $G$ of type $A_n$ which is not $G$-cr over $k$ even though it is $G$-cr over $k_0$. Note that this example works for any $p>0$.
\end{rem}

%% file: conjugacy.tex
\section{A problem of conjugacy classes}
We prove Theorem~\ref{conjugacy}. Here, the key is again the existence of a $1$-dimensional curve $C_1$ such that $T_1(C_1)$ is contained in $\mathfrak{c}_{\textup{Lie}\,({R_u(P_\lambda)})}(K)$ but not contained in  $\textup{Lie}(C_{R_u(P_\lambda)}(K))$ as in the $G_2$ example. Let $G$, $M$, $k$ be as in the hypotheses of the theorem. We keep the notation $q_1, q_2, v, K,  P_\lambda, L_\lambda$ of Section~3. A calculation using the commutator relations (Lemma~\ref{commutation lemma}) shows that
\begin{equation*}
Z(R_u(P_\lambda))=\langle U_{36}, U_{37}, U_{38}, U_{39}, U_{40}, U_{41}, U_{42} \rangle.
\end{equation*}
Let 
$
K_0: = \langle K, Z(R_u(P_\lambda)) \rangle. 
$
It is standard that there exists a finite subset $F=\{z_1, z_2, \cdots, z_{n'}\}$ of $Z(R_u(P))$ such that
$
C_{P_\lambda}(\langle K, F \rangle) = C_{P_\lambda}(K_0).
$
Let
$
\bold{m}:= (q_1, q_2, z_1, \cdots, z_{n'}).
$
Let $n:=n'+2$. For every $x\in k^*$, define
$
\bold{m}(x):= v(x)\cdot \bold{m}\in P_\lambda(M)^{n}.
$

\begin{lem}\label{sizeofthecentralizerofLevi}
$C_{P_\lambda}(K_0)=C_{R_u(P_{\lambda})}(K_0)$.
\end{lem}
\begin{proof}
It is obvious that $C_{R_u(P_{\lambda})}(K_0)\subseteq C_{P_\lambda}(K_0)$. We prove the converse. Let $lu\in C_{P_\lambda}(K_0)$ for some $l\in L_\lambda$ and $u\in R_u(P_\lambda)$. Then $lu$ centralizes $Z(R_u(P_\lambda))$, so $l$ centralizes $Z(R_u(P_\lambda))$, since $u$ does. It suffices to show that $l=1$. Let $l=t \tilde{l}$ where $t\in Z(L_\lambda)^{\circ}=\lambda(k^*)$ and $\tilde{l}\in [L_\lambda, L_\lambda]$. We have
\begin{equation}\label{nontrivialactiononZ}
\langle i, \lambda \rangle = 4 \textup{ for any } i\in \{36,\cdots, 42\}.
\end{equation}
So for any $z\in Z(R_u(P_\lambda))$, there exists $\alpha\in k^{*}$ such that $t\cdot z = \alpha z$. Then we have $\tilde{l}\cdot z = \alpha^{-1} z$. Now define
$
A:=\{\tilde{l}\in [L_\lambda, L_\lambda] \mid \tilde{l} \textup{ acts on }Z(R_u(P_\lambda)) \textup{ by multiplication by a scalar}\}.
$
Then it is easy to see that $A\mathrel{\unlhd} [L_\lambda, L_\lambda]$. Since $[L_\lambda, L_\lambda]\cong SL_7$ and $L_\lambda\cong GL_7$, we have $A=Z([L_\lambda, L_\lambda])$. Therefore we obtain $\tilde{l}\in A=Z([L_\lambda, L_\lambda]) \subseteq \lambda(k^{*})$. So we have $l=c\tilde{l}\in \lambda(k^{*})$. Then we obtain $l\in C_{\lambda(k^{*})}\left(Z(R_u(P_\lambda))\right)$. By (\ref{nontrivialactiononZ}) this implies $l=1$.
\end{proof}

\begin{lem}\label{Pclasses}
$G\cdot \bold{m}\cap P_\lambda(M)^{n}$ is an infinite union of $P_\lambda(M)$-conjugacy classes.
\end{lem}
\begin{proof}
Fix $a'\in k^{*}$. By Lemma~\ref{sizeofthecentralizerofLevi}, we have 
$
C_{P_\lambda}(K_0) = C_{R_u(P_{\lambda})}(K_0) \subseteq  C_{R_u(P_{\lambda})}(K).
$
Then we obtain
\begin{equation}\label{controlling centralizer}
C_{P_\lambda}(v(a')K_0 v(a')^{-1}) = v(a')C_{P_\lambda}(K_0)v(a')^{-1}\subseteq v(a') C_{R_u(P_{\lambda})}(K)v(a')^{-1}.
\end{equation}
Choose $b'\in k^{*}$ such that $\bold{m}(a')$ is $P_\lambda(M)$-conjugate to $\bold{m}(b')$. Then there exists  $m\in P_\lambda(M)$ such that $m\cdot \bold{m}(b')=\bold{m}(a')$. By (\ref{controlling centralizer}), we have 
\begin{equation*}
m v(b') v(a')^{-1}\in C_{P_\lambda}(v(a') K_0 v(a')^{-1})\subseteq v(a') C_{R_u(P_{\lambda})}(K)v(a')^{-1}.
\end{equation*}
By Proposition~\ref{positive proposition}, we have 
\begin{alignat*}{2}
v(a') ^{-1} m v(b') &=\prod_{i=1}^{7}\epsilon_i(a)\prod_{i=8}^{14}\epsilon_i(b)\prod_{i=15}^{28}\epsilon_i(c)\prod_{i=29}^{35}
\epsilon_i(a+b+c)\prod_{i=36}^{42}\epsilon_i(a_i), \textup{ for some } a,b,c,a_i\in k. 
\end{alignat*}
This yields
\begin{equation*}
m=\prod_{i=1}^{7}\epsilon_i(a+a'+b')\prod_{i=8}^{14}\epsilon_i(b)\prod_{i=15}^{28}\epsilon_i(c)\prod_{i=29}^{35}
\epsilon_i(a+b+c)\prod_{i=36}^{42}\epsilon_i(b_i), \textup{ for some } a,b,c,b_i\in k. 
\end{equation*} 
But $m\in P_\lambda(M)$, so
$
a+a'+b' = 0, b=0, c=0, a+b+c=0. 
$
Hence we have
$
a'=b'. 
$
Thus we have shown that if $a'\ne b'$, then $\bold{m}(a')$ is not $P_\lambda(M)$-conjugate to $\bold{m}(b')$. So, in particular, $G\cdot \bold{m}\cap P_\lambda(M)^{n}$ is an infinite union of $P_\lambda(M)$-conjugacy classes.
\end{proof}
We need the next result~\cite[Lem.~4.4]{Lond}. We include the proof to make this paper self-contained.
\begin{lem}\label{Daniel}
$G\cdot \bold{m}\cap P_\lambda(M)^{n}$ is a finite union of $M$-conjugacy classes if and only if it is a finite union of $P_\lambda(M)$-conjugacy classes.
\end{lem}
\begin{proof}
Pick $\bold{m_1}, \bold{m_2}\in G\cdot \bold{m}\cap P_\lambda(M)^{n}$ such that $\bold{m_1}$ and $\bold{m_2}$ are in the same $M$-conjugacy class of $G\cdot \bold{m}\cap P_\lambda(M)^{n}$. Then there exists $m\in M$ such that $m\cdot \bold{m_1} = \bold{m_2}$. Let $Q=m^{-1} P_\lambda(M) m$. Then we have $\bold{m_1} \in (P_\lambda(M)\cap Q)^{n}$. Now let $S$ be a maximal torus of $M$ contained in $P_\lambda(M)\cap Q$. Since $S$ and $m^{-1} S m$ are maximal tori of $Q$, they must be $Q$-conjugate. So there exists $q\in Q$ such that 
\begin{equation}\label{S-conjugacy}
q S q^{-1} = m^{-1} S m.
\end{equation}
Since $Q=m^{-1} P_\lambda(M) m$, there exists $p\in P_\lambda(M)$ such that $q= m^{-1} p m$. Then from (\ref{S-conjugacy}), we obtain 
$
p m S m^{-1} p^{-1} = S.
$ 
This implies
$
m^{-1} p^{-1} \in N_M(S).
$
Fix a finite set $N\subseteq N_M(S)$ of coset representatives for the Weyl group $W=N_M(S)/S$. Then we have
\begin{equation*}
m^{-1} p^{-1} = n s \textup{ for some } n\in N, s\in S.
\end{equation*}
So we obtain
$
\bold{m_1} = m^{-1} \cdot \bold{m_2} = (n s p) \cdot \bold{m_2}\in (n P_\lambda(M))\cdot \bold{m_2}.
$
Since $N$ is a finite set, this shows that a $M$-conjugacy class in $G\cdot \bold{m}\cap P_\lambda(M)^{n}$ is a finite union of $P_\lambda(M)$-conjugacy classes. The converse is obvious. 
\end{proof}
\begin{proof}[Proof of Theorem \ref{conjugacy-counterexample}]
By Lemma~\ref{Pclasses} and Lemma~\ref{Daniel}, we conclude that $G\cdot \bold{m}\cap P_\lambda(M)^{n}$ is an infinite union of $M$-conjugacy classes. Now it is evident that $G\cdot \bold{m}\cap M^{n}$ is an infinite union of $M$-conjugacy classes. 
\end{proof}

%% file: acknowledgements.tex
\section*{Acknowledgements}
This research was supported by a University of Canterbury Master's Scholarship and Marsden Grant UOC1009/UOA1021. The author would like to thank Benjamin Martin and G\"{u}nter Steinke for  helpful discussions. He is also grateful for detailed comments from J.P. Serre and an anonymous referee.

%% file: Appendix.tex
\section*{Appendix}
\begin{table}[h!]
\scalebox{0.7}{
\begin{tabular}{llll}
\rootsG{1}{1}{0}{0}{1}{1}{1}{0}&\rootsG{2}{1}{1}{1}{1}{1}{0}{0}&\rootsG{3}{1}{0}{1}{1}{2}{1}{1}&\rootsG{4}{1}{0}{0}{1}{2}{2}{1}\\
&&&\\
\rootsG{5}{1}{1}{1}{2}{2}{1}{0}&\rootsG{6}{1}{0}{1}{1}{2}{2}{1}&\rootsG{7}{1}{1}{2}{2}{2}{1}{1}&\rootsG{8}{1}{0}{0}{0}{0}{0}{0}\\
&&&\\
\rootsG{9}{1}{0}{0}{0}{1}{0}{0}&\rootsG{10}{1}{0}{1}{1}{1}{1}{0}&\rootsG{11}{1}{0}{0}{1}{2}{1}{1}&\rootsG{12}{1}{1}{2}{2}{2}{2}{1}\\
&&&\\
\rootsG{13}{1}{1}{1}{2}{3}{2}{1}&\rootsG{14}{1}{1}{2}{3}{3}{2}{1}&\rootsG{15}{1}{0}{0}{1}{1}{0}{0}&\rootsG{16}{1}{0}{0}{0}{1}{1}{0}\\
&&&\\
\rootsG{17}{1}{0}{1}{1}{1}{0}{0}&\rootsG{18}{1}{0}{0}{0}{1}{1}{1}&\rootsG{19}{1}{0}{0}{1}{2}{1}{0}&\rootsG{20}{1}{1}{1}{1}{1}{1}{0}\\
&&&\\
\rootsG{21}{1}{0}{1}{1}{1}{1}{1}&\rootsG{22}{1}{1}{1}{1}{2}{1}{1}&\rootsG{23}{1}{1}{2}{2}{2}{1}{0}&\rootsG{24}{1}{1}{1}{2}{2}{1}{1}\\
&&&\\
\rootsG{25}{1}{0}{1}{2}{2}{2}{1}&\rootsG{26}{1}{1}{1}{2}{2}{2}{1}&\rootsG{27}{1}{0}{1}{2}{3}{2}{1}&\rootsG{28}{1}{1}{2}{2}{3}{2}{1}\\
&&&\\
\rootsG{29}{1}{0}{0}{1}{1}{1}{1}&\rootsG{30}{1}{0}{1}{1}{2}{1}{0}&\rootsG{31}{1}{1}{1}{1}{2}{1}{0}&\rootsG{32}{1}{1}{1}{1}{1}{1}{1}\\
&&&\\
\rootsG{33}{1}{0}{1}{2}{2}{1}{0}&\rootsG{34}{1}{0}{1}{2}{2}{1}{1}&\rootsG{35}{1}{1}{1}{1}{2}{2}{1}&\rootsG{36}{2}{0}{1}{2}{3}{2}{1}\\
&&&\\
\rootsG{37}{2}{1}{1}{2}{3}{2}{1}&\rootsG{38}{2}{1}{2}{2}{3}{2}{1}&\rootsG{39}{2}{1}{2}{3}{3}{2}{1}&\rootsG{40}{2}{1}{2}{3}{4}{2}{1}\\
&&&\\
\rootsG{41}{2}{1}{2}{3}{4}{3}{1}&\rootsG{42}{2}{1}{2}{3}{4}{3}{2}&\rootsG{43}{0}{1}{0}{0}{0}{0}{0}&\rootsG{44}{0}{0}{1}{0}{0}{0}{0}\\
&&&\\
\rootsG{45}{0}{0}{0}{1}{0}{0}{0}&\rootsG{46}{0}{0}{0}{0}{1}{0}{0}&\rootsG{47}{0}{0}{0}{0}{0}{1}{0}&\rootsG{48}{0}{0}{0}{0}{0}{0}{1}\\
&&&\\
\rootsG{49}{0}{1}{1}{0}{0}{0}{0}&\rootsG{50}{0}{0}{1}{1}{0}{0}{0}&\rootsG{51}{0}{0}{0}{1}{1}{0}{0}&\rootsG{52}{0}{0}{0}{0}{1}{1}{0}\\
&&&\\
\rootsG{53}{0}{0}{0}{0}{0}{1}{1}&\rootsG{54}{0}{1}{1}{1}{0}{0}{0}&\rootsG{55}{0}{0}{1}{1}{1}{0}{0}&\rootsG{56}{0}{0}{0}{1}{1}{1}{0}\\
&&&\\
\rootsG{57}{0}{0}{0}{0}{1}{1}{1}&\rootsG{58}{0}{1}{1}{1}{1}{0}{0}&\rootsG{59}{0}{0}{1}{1}{1}{1}{0}&\rootsG{60}{0}{0}{0}{1}{1}{1}{1}\\
&&&\\
\rootsG{61}{0}{1}{1}{1}{1}{1}{0}&\rootsG{62}{0}{0}{1}{1}{1}{1}{1}&\rootsG{63}{0}{1}{1}{1}{1}{1}{1}&\\
&&&\\
\end{tabular}
}
\caption{The set of positive roots of $G=E_7$}
\label{Roots of E_7}

\end{table}